\documentclass[12pt]{amsart}

\usepackage[usenames,dvipsnames]{xcolor}

\usepackage{fancyhdr}
\usepackage{amsmath,amsthm,amsfonts,graphicx,amssymb,amscd}
\usepackage[all,cmtip]{xy}
\usepackage{graphicx, overpic, float}
\usepackage[mathscr]{euscript}
\usepackage{hyperref}
\usepackage{enumerate}
\usepackage{lipsum}

\newtheorem{thm}{Theorem}[section]
\newtheorem*{jthm}{Theorem}
\newtheorem{conj}{Conjecture}

\newtheorem{cor}[thm]{Corollary}
\newtheorem{lem}[thm]{Lemma}

\newtheorem{prop}[thm]{Proposition}
\newtheorem{proposition}[thm]{Proposition}
\theoremstyle{definition}
\newtheorem{defn}[thm]{Definition}

\newcommand{\T}{\mathscr{T}}

\newcommand{\R}{\mathbb{R}}
\newcommand{\C}{\mathbb{C}}

\newcommand{\M}{\mathscr{M}}

\newcommand{\ep}{\epsilon}

\newcommand{\om}{\Omega}
\newcommand{\rt}{\rightarrow}

\usepackage{fancybox}

\usepackage{ulem}

\renewcommand{\emph}[1]{{\it #1}}

\usepackage{color}

\title{On domains biholomorphic to Teichm\"{u}ller spaces}

\author{Subhojoy Gupta}
\address{Department of Mathematics, Indian Institute of Science, Bangalore 560012, India.}
\email{subhojoy@math.iisc.ernet.in}

\author{Harish Seshadri}
\email{harish@math.iisc.ernet.in}

\begin{document}
\setcounter{tocdepth}{4}
\maketitle

\begin{abstract}
We prove that  the Teichm\"{u}ller space  $\mathscr{T}$ of a closed surface of genus $g \ge 2$ cannot be biholomorphic to any  domain which is  locally strictly convex at some boundary point.
\end{abstract}



\section{Introduction}
A classical result of L. Bers asserts that the Teichm\"{u}ller space $\T$ of a closed surface of genus $g\geq 2$ can be realized as a bounded domain in $\mathbb{C}^{3g-3}$.
While it is known that $\T$, endowed with the complex structure induced by this embedding, is pseudoconvex (see for example \cite{Krushkal}, \cite{Yeung}, \cite{Shiga}),  there is only a partial understanding of the biholomorphism-type of $\T$. In particular, the following is a folklore conjecture recently proved by V. Markovic in \cite{Mark}:

\begin{conj}\label{main-conj} The Teichm\"{u}ller space of a closed surface of genus $g\geq 2$  cannot be biholomorphic to a bounded convex domain.
\end{conj}

 In this paper we address the finer question of \textit{local convexity} at individual boundary points. We say that a domain $\om \subset \C^n$ is locally convex at $p \in \partial \om$ if $\om \cap B(p,r)$ is convex for some $r>0$, where $B(p,r)$ denotes the Euclidean ball with center $p$ and radius $r$.  We raise the following stronger conjecture:

\begin{conj}\label{main-conj2} The Teichm\"{u}ller space of a closed surface of genus $g\geq 2$  cannot be biholomorphic to a bounded domain $\om \subset \C^{3g-3}$ that is {locally convex} at some boundary point.
\end{conj}

Recall that a domain $\Omega \subset \C^N$  is said to be strictly convex if $\om$ is convex and $\partial \om$ does not contain any line segment. As above, we say that a domain $\om$ is   locally strictly convex at $p\in \partial \om $  if there exists $r>0$  such that $ \Omega \cap B(p,r)$ is strictly convex. \vspace{2mm}

Our main result is as follows:

\begin{thm}\label{thm1}
The Teichm\"{u}ller space $\T$ of a closed surface of genus $g\geq 2$ cannot be biholomorphic to a bounded domain $\om \subset \C^{3g-3}$ which is locally strictly convex at some boundary point.
\end{thm}

{\bf Remark:}  In fact,  minor modifications of the proof of Theorem \ref{thm1} yield a more general statement:  $\T$ is not biholomorphic to a bounded domain $\om$ with a locally strictly {\it convexifiable} boundary point.  By definition, if $p \in \partial \om$ is such a point, then there exists $r >0$ and a holomorphic embedding $F: \om \cap B(p,r) \rt \C^{3g-3}$ such that $F(\om \cap B(p,r))$  is strictly convex.  \vspace{3mm}

 As an application of Theorem \ref{thm1}, we give a new proof of the following result of S-T. Yau:

\begin{thm}\label{thm3}
The Teichm\"{u}ller space $\T$ of a closed surface of genus $g\geq 2$ cannot be biholomorphic to a bounded domain with $C^2$-smooth boundary.
\end{thm}


As described in S-T. Siu's survey \cite{UnifSiu},  Conjecture \ref{main-conj}  fits in the context of the work of S. Frankel (\cite{Frankel}) who proves that any bounded convex domain with a co-compact group of automorphisms acting freely and properly discontinuously  must be biholomorphic to a bounded symmetric domain.  As for $\T$,  it is known that there is a discrete subgroup $\Gamma \subset Aut(\T)$ acting freely on $\T$ such that the quotient $\T/\Gamma$ has finite Kobayashi volume. However $\T/\Gamma$ is noncompact and a generalization of Frankel's result to the case of convex domains with  finite-volume quotients is not known. 
\vspace{2mm}


A key ingredient in Frankel's proof is complex rescaling along automorphism orbits. As  Frankel points out in his paper (see \S7 of \cite{Frankel}),  this technique can be employed even if one replaces the convexity assumption by local convexity at an orbit accumulation point.  This, and localization results in complex analysis such as the Wong-Rosay theorem  \cite{Wong}, \cite{Rosay}, motivated us to raise
Conjecture \ref{main-conj2}.  
\vspace{2mm}



\textbf{ Outline of proof:}  The proof is broadly  based on K.-T. Kim's  proof that the Bers embedding of $\T$ is not convex (see \cite{Kim-Bers}).    
We argue by contradiction and suppose that $\T$ is biholomorphic to a bounded domain $\om$ with a locally strictly convex boundary point. 
The  main technical result in this paper asserts that  any such point is an orbit accumulation point of $Aut(\T)$. In Kim's paper,  C. McMullen's result about the density of cusps in the Bers boundary  \cite{McM} was used to prove that {\it every} point in the Bers boundary is an orbit accumulation point. Unfortunately the latter property may not be preserved under biholomorphisms of domains as  biholomorphism may not extend in an absolutely continuous (or even continuous) manner to  closures of domains.  \vspace{2mm}
     
By choosing such an orbit accumulation point to be an Alexandroff-smooth point (see  Theorem \ref{alex-thm})  and applying a rescaling argument due to Kim-Krantz-Pinchuk (described in the Appendix), one obtains a convex domain with a {\it non-discrete}  automorphism group.  This contradicts a basic fact due to H. Royden \cite{Royd}  that $\text{Aut}(\T)$ is discrete.\vspace{2mm}

The proof of the  orbit accumulation property is based on two basic results in Teichm\"{u}ller theory: 

First, the abundance of holomorphic \textit{Teichm\"{u}ller disks} in $\T$  that are both totally geodesic in the Teichm\"{u}ller metric, and complex geodesics in the Kobayashi metric; the two metrics coincide by the work of H. Royden in \cite{Royd}.

Second, the ergodicity of the Teichm\"{u}ller geodesic flow due to H. Masur \cite{Masur} and W. Veech \cite{Veech}.  Their work implies, in particular, that for any Teichm\"{u}ller disk,  almost every radial ray gives rise to a geodesic ray in $\T$ that projects to a dense set in  moduli space  $\M_g:= \T / Aut(\T) $.  In particular, there is a sequence of  points along such rays that recur to any fixed compact set in $\M_g$.

Given these facts, the proof involves an elementary but delicate analysis of the boundary behaviour of  holomorphic functions on the unit disc in $\C$. We choose a point $q \in \om$ which is close to the strictly convex point $p \in \om$ and take a complex geodesic $\phi : \Delta \rt \om$ with $\phi(0)=q$.  The idea is to use a pluriharmonic ``barrier" function, namely the height from a supporting hyperplane, to prove the existence of  a positive measure set of directions in $\partial \Delta$ for which the radial limits of $\phi$ are close to $p$.  One can then apply the Masur-Veech ergodicity result to infer the existence of an orbit point close to $p$.  \vspace{2mm}



\textbf{Note:} After the first draft of the this paper was circulated, we came to know of  the preprint of V. Markovic \cite{Mark} where it is proved that  Kobayashi and Caratheodory metrics are not equal on $\T$ . By a result of L. Lempert (and its extension due to  H. Royden - P. M. Wong) this implies Conjecture \ref{main-conj}. 
However, the techniques employed in this paper are completely different and  our main result (Theorem \ref{thm1}) does not follow, to the best of our knowledge, from the work of Markovic if we do not assume the main technical result of this paper about locally strictly convex points being orbit accumulation points. 



\vspace{2mm}


\textbf{Acknowledgements:} We are grateful to  Gautam Bharali and Kaushal Verma for many illuminating conversations.
 The first author thanks the hospitality and support of the center of excellence grant `Center for Quantum Geometry of Moduli Spaces' from the Danish National Research Foundation (DNRF95) during May-June 2016. He wishes to acknowledge that the research leading to these results was supported by a Marie Curie International Research Staff Exchange Scheme Fellowship within the 7th European Union Framework Programme (FP7/2007-2013) under grant agreement no. 612534, project MODULI - Indo European Collaboration on Moduli Spaces.

\section{Preliminaries}

\subsection{Teichm\"{u}ller space and Teichm\"{u}ller disks}

We briefly recall some basic facts and definitions in Teichm\"{u}ller theory that shall be relevant to this paper; see \cite{Lehto}, \cite{Ahlfors}, \cite{Hubbard} for further details and references.

For a closed surface $S$ of genus $g\geq 2$, the Teichm\"{u}ller space $\T$ is the space of marked Riemann surfaces:
\begin{center}
$\T  = \{ (f, X) \vert \text{ } X \text{ is a Riemann surface and } f:S\to X \text{ is a homeomorphism} \}/\sim$
\end{center}
where $(f,X) \sim (g,Y)$ if there is a biholomorphism $h:X\to Y$ that is homotopic to $g\circ f^{-1}$.

The Teichm\"{u}ller metric  on $\T$  is defined as
\begin{equation}\label{teich-dist}
d_\T((f,X), (g,Y)) = \frac{1}{2} \inf\limits_{h}  \ln K(h)
\end{equation}
where  the infimum varies over all quasiconformal homeomorphisms $h:X\to Y$ that are homotopic to $g\circ f^{-1}$, and  $K(h)$ is the quasiconformal distortion of the map $h$.

Teichm\"{u}ller in fact showed that the infimum in (\ref{teich-dist}) is achieved by a \textit{Teichm\"{u}ller map} that is determined by a holomorphic quadratic differential $q$ on $X$. Such differentials in fact span the cotangent space $T_X^\ast \T$ to Teichm\"{u}ller space at the point $X$, and any point $X$ and direction $q$ determines a \textit{Teichm\"{u}ller geodesic ray}.

The following results from Teichm\"{u}ller theory  is of crucial  importance to this paper.\vspace{2mm}

First, we have:

\begin{thm}[Masur \cite{Masur}, Veech \cite{Veech}]  The Teichm\"{u}ller geodesic flow is ergodic: that is, for almost every $X\in \T$ and $q\in T_X^\ast \T$, the  projection of the Teichm\"{u}ller geodesic ray to the cotangent bundle of moduli space $\M_g$ equidistributes.
\end{thm}

Second,  a \textit{Teichm\"{u}ller disk} is an isometric embedding
\begin{center}
$\mathcal{D} :(\Delta, \rho_\Delta)  \to (\T, d_\T)$
\end{center}
which is also holomorphic. Here $\rho_\Delta$ is the distance function of the Poinc'{a}re  metric on $\Delta$. See for example, \S9.5 of \cite{Lehto}, for their construction and properties.

Such embeddings are abundant:  there is a Teichm\"{u}ller disk through any given  point $X\in \T$ and direction $q\in T_X^\ast \T$.
Moreover,  the image of the radial rays $R_\theta = \mathcal{D}(\cdot, \theta)$ for a fixed angle $\theta$ are {Teichm\"{u}ller geodesic rays} in $\T$.

We shall use the following consequence of the ergodicity of the Teichm\"{u}ller geodesic flow. In what follows, a Teichm\"{u}ller geodesic ray is \textit{dense} if its projection to $\M_g$ is a dense set.

\begin{prop}\label{key-prop}  Given $X\in \T$,  there exists a  Teichm\"{u}ller disk $\mathcal{D} :\Delta \to \T$  with  $\mathcal{D}(0)=X$ and a  full measure subset $\Xi \subset \partial \Delta$  such  that  if  $e^{i\theta} \in \Xi $, then the Teichm\"{u}ller ray $R_\theta \subset D$ is dense in the above sense.
\end{prop}
\begin{proof}
By the ergodicity theorem, we have that for almost every point $X\in \T$ and almost every direction $q\in T_X^\ast\T$,  the Teichm\"{u}ller geodesic ray from $X$ in the direction $q$ equidistributes in the tangent bundle $\mathcal{Q} \to \T$, and in particular, projects to a dense set in $\M_g$. From the same work of Masur and Veech, it follows that the subset of \textit{uniquely-ergodic} directions $q\in T_X^\ast\T$ is of full measure. Moreover, from Theorem 2 of \cite{Masur0} two Teichm\"{u}ller rays in the same uniquely-ergodic direction are (strongly) asymptotic, and hence at \textit{any} given basepoint $X$, the set of directions $E$ which yield a dense ray is of full measure.

The unit sphere of directions $S^{6g-7}$ comprising holomorphic quadratic differentials on $X$ of unit norm is foliated by circles corresponding to directions given by Teichm\"{u}ller disks; by Fubini's theorem almost every such circle will intersect $E$ in a set of full measure. In particular, there exists a Teichm\"{u}ller disk and a full measure set of directions $\Xi$ in the unit circle such that each corresponding radial ray gives rise to a dense Teichm\"{u}ller geodesic ray. \end{proof}

\subsection{Complex structure on $\T$:}

The complex structure on Teichm\"{u}ller space is inherited by the embedding
\begin{center}
$\mathcal{B}_X:\T \to \mathbb{C}^{3g-3}$
\end{center}
introduced by Bers in \cite{Bers}. (Though the embedding depends on a choice  of a fixed Riemann surface $X \in \T$, varying $X$ produces biholomorphically equivalent domains.)

The image of $\mathcal{B}_X$ is a bounded domain.

On any bounded domain $\om$ in $\C^N$,  the \textit{infinitesimal Kobayashi metric} is defined by the following norm for a tangent vector $v$ at a point $X\in \om$:
\begin{equation}\label{kob}
K(X,v) = \inf_{h:\Delta \to \om} \frac{\lvert v\rvert}{\lvert h^\prime(0)\rvert}
\end{equation}
where over all holomorphic maps $h:\Delta \to \om$ such that $h(0) = X$ and $h^\prime(0)$ is a multiple of $v$.

The \textit{Kobayashi metric} $d^K$ on $\om$  is then the distance defined in the usual way: one first defines lengths of piecewise $C^1$ curves in $\om$ using the above norm and then takes the infimum of lengths of curves joining two given points to get the distance between them. \vspace{2mm}

We list some elementary properties of the Kobayashi metric  $d^K$ on a bounded domain $\om$ in $\C^N$.    
\begin{prop}\label{koba}
(1)  Biholomorphisms of $\om$ are isometries for $d^K$.

(2) If $\Omega \subset \Omega^\prime$ then their respective Kobayashi metrics  satisfy $d_{K}^\prime \leq d_K$ on the smaller domain $\Omega$.

(3) Let $d^e$ denote  Euclidean distance on $\C^N$. There exists $C=C(\om)$ such that $d^e \le C d^K$ on $\om$.
\end{prop}

\begin{proof}[Sketch of the proof]

(1) and (2) are immediate from the definition (\ref{kob}), and (3) then follows from the comparison obtained by embedding $\om$ in a round ball of radius $C$ in $\C^N$.
\end{proof}

The Kobayashi metric on the image of the Bers embedding gives rise to the Kobayashi metric $d^K$ on $\T$.  The following fundamental result is due to Royden:

\begin{jthm}[Royden \cite{Royd}] \label{cor-royd}  
(i) The Kobayashi and Teichm\"{u}ller metrics on $\T$ coincide, that is,  $d^T=d^K$.

(ii) The group of biholomorphisms of $\T$ is the mapping class group.
\end{jthm}

Here the mapping class group $MCG(S)$, defined by 
\begin{center}
$\mathrm{MCG}(S) = \text{Homeo}^+(S)/\text{Homeo}^+_0(S),$
\end{center}
is the group of orientation-preserving homeomorphisms of $S$ quotiented by the subgroup of those homotopic to the identity.
This is a discrete group in the topology induced from the compact-open topology on $\text{Homeo}^+(S)$.
$\mathrm{MCG}(S)$ acts properly discontinuously on $\T$ by the action $\phi \cdot (f,X) \to (f\circ \phi^{-1}, X)$ and the quotient is the Riemann moduli space $\M_g$.

Finally, we note that the Teichm\"{u}ller disks introduced in \S2.1 are complex geodesics in the following sense:

\begin{defn}\label{cg}
Let $\om \subset \C^N$ be a bounded domain. A {\it complex geodesic} in $\om$ is an isometric, holomorphic embedding $\phi:(\Delta, \rho_\Delta)  \rt (\om, d^K)$, where $\rho_\Delta$ denotes the Poincare metric on $\Delta$ and $d^K$ is the Kobayashi metric.
\end{defn}

\section{Orbit accumulation points}
We prove the main technical result of this paper in this section:

\begin{prop}[Shadowing orbit] \label{mcg}
Let $\om\subset \C^{3g-3}$ be a bounded domain that is biholomorphic to the Teichm\"{u}ller space $\T$. 
If  $p \in \partial \Omega$ is a locally strictly convex boundary point,  then $p$ is an accumulation point of an orbit of $Aut(D)=MCG(S)$. 
\end{prop}

\subsection{Basic Setup}

We shall use the following notation:\vspace{2mm}

$\bullet$ $\Delta := \{z \in \C \ : \ \vert z \vert=1\}$ \vspace{2mm}


$\bullet$  For  $r>0$, $B(p,r)$ denotes the closed Euclidean ball of radius $r$ and center $p$.\vspace{2mm}

$\bullet$ Let $\om \subset \C^N$ be a domain and $p \in \partial \om$ is a locally strictly convex point. Suppose that
$U=\om \cap B(p,r)$ is strictly convex for some $r>0$.  \vspace{2mm}

$\bullet$  $V_p$  denotes a supporting hyperplane for $U$ at $p$. \vspace{2mm}

\begin{defn}[Height function]\label{hf} 
We define the \textit{height} $h(q)$ of any point $q\in \Omega$ to be the Euclidean distance between $q$ and  $V_p$, i.e.,  if $\nu \in (V_p-p)^\perp \subset  \C^N$ is a unit normal to $V_p$ pointing into $\om$, then
$$h(q):= \langle q-p, \nu \rangle.$$

 This defines a pluriharmonic function $h_p:\C^N \to \mathbb{R}$. 
\end{defn}

$\bullet$ For $\delta >0$ let ${\om}_{\leq \delta}:=h^{-1}((-\infty,\delta])$ be the $\delta$-sublevel set of the height function and let ${\om}^0_{\leq \delta}$ be the connected component of $\om_{\leq \delta}$  containing the point $p$.  \vspace{2mm}

We have the following elementary observation:

\begin{lem}\label{nest-lem}
As $\delta \to 0$, $\overline{\om}_{\leq \delta}^0$ converges in the Hausdorff topology   to $\{p\}$ .
\end{lem}

\begin{proof}
We just have to check that given $\ep >0$, there is a $\delta_0 >0$ such that $\overline{\om}_{\leq \delta} \subset B(p,\ep)$  for all $\delta < \delta_0$. 

If not, there is an $\ep>0$ and a sequence  $q_i \in \overline{\om}_{\leq \delta_i}^0 \setminus  B(p,\ep) $ where $\delta_i \rt 0$ as $i \rt \infty$.

Since $\delta_i \to 0$,  after passing to a subsequence,  $\{q_i\}_{i \ge 1}$  converges to a point $q\in  C = V_p \cap \partial \om^0_{\le \delta_i}$. The strict convexity assumption implies that   $V_p \cap \partial \om^0_{\le \delta_i}= \{p\}$.  Hence, for $i$ large enough,
$q_i \in B(p,\ep)$, which is a contradiction. \end{proof}




Let $\phi:\Delta \to \om$ be a complex geodesic such that $\phi(0) \in {\om}^0_{\leq  \delta}$.\vspace{2mm}

Also, let 
\begin{equation}\label{harm-F}
F:\Delta \to \mathbb{R}_+
\end{equation}
 be the harmonic function  $F: = h\circ \phi$. Note that 
\begin{equation}\label{eq1}
F(0) <\delta
\end{equation}

\medskip


\subsection{Positive measure limits}

The goal of this section is to show:

\begin{prop}\label{main-prop}
There is a positive measure subset $E_1 \subset \partial \Delta$  such that for each $\theta\in E_1$, the geodesic ray $\gamma_\theta = \{\phi(re^{i\theta}) \  : \ 0<r<1\}$  eventually lies in  ${\om}^0_{\leq 2\delta}$.
\end{prop}

\medskip

Let $V\subset \Delta$ be the component of the pre-image set $\phi^{-1} (  {\om}^0_{\leq \delta})$ that contains $0$.
By changing $\delta$ slightly we can assume that $\partial V \subset \partial \phi^{-1} ( {\om}^0_{\leq \delta})$ is a submanifold of $\Delta$, i.e.,  is a union of homeomorphic copies of circles and open intervals $I$ such that $\overline I  - I \subset \partial \Delta$. 
\begin{lem}\label{sc}
$V$ is simply connected.
\end{lem}
\begin{proof}
 Since $V$ is a planar domain, it is homeomorphic to a disc with (possibly infinitely many) punctures. In particular there is a nontrivial element of $\pi_1(V)$ which is represented by a smooth embedding  of $S^1$ in $V$. We denote the image of this embedding by $\sigma$.   By the Jordan curve theorem $\C - \sigma$ has exactly one connected component $W$ with compact closure. Note that $W \cup \sigma$ is homeomorphic to a closed disc.

As observed above, $\partial V$ comprises of arcs with limit points in $\partial \Delta$ and closed curves (embedded circles) in the interior $\Delta$. 
Assume $W$ is not contained in $V$, that is,  $W \cap \partial V \neq \emptyset$. However $W$ cannot intersect any arc component of $\partial V$ as then that arc would intersect $\sigma \subset V$, which is a contradiction.  Thus $W$ contains an embedded circle  $\gamma \subset \partial V$.  However $\gamma$ is the boundary of a topological disc $D \subset \Delta$.  Since the harmonic function $F=h \circ \phi$ is the constant $\delta$ on $\gamma$, we would conclude that $F \equiv \delta$ on $D$ and hence on $\Delta$. This contradicts our assumption that $F(0) < \delta$. 

Hence $W \subset V$. But this contradicts our assumption that $\sigma$ is a nontrivial loop in $V$. 
\end{proof}
 
 Our analysis will involve understanding the boundary behavior of $V$. 
 The following classical results is the basis of our analysis:

 \begin{thm}[Fatou]\label{fatou}
 For any bounded holomorphic or harmonic function $f:\Delta \to \mathbb{C}$ there exists a measurable subset $X \subset \partial \Delta$ of full measure such that for any $\theta \in X$, the radial limit of $f$ exists. We denote this limit by $$ \hat{f}(e^{i\theta}) = \lim\limits_{r\to 1} f(re^{i\theta})$$. 
  \end{thm}
 
\textit{Remark.} These radial limits correspond precisely to ``accessible points'' of the boundary of the image domain when $f$ is a uniformizing Riemann map (see Theorem 4.18 of \cite{BrannerBook}.\vspace{2mm}

 
 Finally, we recall the notion of ``harmonic measure" in the following theorem that summarizes some of its properties (see \cite{Ransford} for details).  
  
 \begin{lem}\label{harm-thm}
(i)  Let $U \subset \C$ be a domain such that $\partial U$ is not a polar set. Fix a $z \in U$. Then  there is a unique Borel measure $\omega(z,U)$ on $\partial U$, called the harmonic measure of $U$ with respect to $z$,  with the property that for any continuous function $f: \partial U \to \mathbb{R}$, the harmonic extension $F:U \to \mathbb{R}$ is given by :
 \begin{equation*}
 F(z) = \displaystyle\int\limits_{\partial U}  f(\eta) \omega(z, U).
 \end{equation*}
 
 (ii)  Suppose that $U$ is simply-connected and $f: \Delta \rt U$ is a conformal map with $f(0)=z$.  Let $Y  \subset \partial \Delta$ be the set of full measure where the radial limits of $f$ exist and let $Z= \hat {f} (Y) \subset \partial U$.  If $E \subset Z$ is a Borel set then $\hat{f}^{-1}(E)$ is a Borel set and 
  $$\omega(0,\Delta)( {\hat f}^{-1}(E)) = \omega(z, U)( E).$$
 
 (iii) For any $z \in \Delta$, $\omega(0, \Delta)$ is absolutely continuous with respect to the Lebesgue measure on $\partial \Delta$.  
 \end{lem}

We can now prove:
 
\begin{lem}\label{hh} Let $V \subset \Delta$ be the component of $\phi^{-1}(\om^0_{\leq \delta})$ containing $0$. Then the Lebesgue measure $m(\partial V \cap \partial \Delta)>0$.
\end{lem}
\begin{proof}

     Let  $E=\partial V \cap \partial \Delta$ and suppose that suppose that $m(E)=0$.  Note that Lemma \ref{harm-thm} applies as $\partial V$ is not a polar set: not every harmonic function on $V$ is constant.

     
     A classical monotonicity estimate of T. Carleman (see Theorem 4.3.8 of \cite{Ransford}) then
implies that 
$$ \omega(z,V)(E) \le \omega (z, \Delta)(E) \text{ for any }z\in V.$$
If $m(E)=0$, then $\omega (z, \Delta)(E)=0$ by the final statement  of Lemma \ref{harm-thm}. Hence $\omega(z,V)(E)=0$ by the above inequality. 

By Lemma  \ref{sc}, there exists a uniformizing biholomorphism $f: \Delta \rt V$  with $f(0)=z$.  Fatou's theorem (Theorem \ref{fatou}) gives us a set of full measure $X_0 \subset \partial \Delta$  such that $f$ has radial limits at every $e^{\theta} \in X_0$. 

Also, Fatou's theorem applied to the harmonic function $F_1 = F\circ f$ where $F$ is the harmonic height function in \eqref{harm-F}, yields a full measure set $X_1 \subset \partial \Delta$ such that $F_1$ has  radial limits at every $e^{i\theta} \in X$.

Let $X = X_0 \cap X_1$ be the full measure set obtained as their intersection. 
Let $\hat f: X \rt  \partial V$ be the boundary map defined by the radial limits.

Recall that arcs of $\partial V$ in the interior of $\Delta$ are the level sets of the harmonic function $F$ as in \eqref{harm-F}, for the value $\delta$. Thus for any point  $q=e^{i\theta} \in X$ such that  $\hat f(q) \in \partial V \cap \Delta$, the harmonic function $F_1 = F \circ f$ has a radial limit $\delta$ in the radial direction at angle $\theta$. 

 Since $\partial V = (\partial V \cap \Delta) \sqcup E$, for any other point $q^\prime \in X$, we have $\hat f (q^\prime) \in E$.
 That is, if  $E_1 = E \cap \hat f(X)$, then $q^\prime \in {\hat f}^{-1}(E_1)$.
 
 But by the conformal invariance of harmonic measure as decsribed in (ii) of Lemma \ref{harm-thm}, we have:
$$ \omega(0,\Delta)( {\hat f}^{-1}(E_1)) = \omega(z, V)( E_1) \le \omega(z,V)(E) =0.$$


Thus, for the harmonic function $F_1:\Delta \to \mathbb{R}$ almost every radial limit is $\delta$.

The Poisson integral formula  (see (i) of Lemma \ref{harm-thm}) then implies that $F_1 \equiv \delta$ on $\Delta$, which is a contradiction to \eqref{eq1}. 
\end{proof}

We will use  Fatou's theorem, together with the following variant of Lusin's theorem for measurable functions:

\begin{lem}\label{inf}  Let $f, X$ be as described in Theorem \ref{fatou}. For any $\ep>0$ there is an $r_0>0$ and a positive measure subset $E_0 \subset \partial V   \cap X$   with the following property: for any  $r_0<r<1$ and  $e^{i\theta} \in E_0$ we have 
\begin{center}
$\lvert \hat f(e^{i\theta}) - f(re^{i\theta}) \lvert \le  \ep$
\end{center}

\end{lem}

\begin{proof}
Fix an $\ep>0$. For any $n>1$, let
$$X_n=  \{  \theta \in X \ : \  \lvert f(e^{i\theta)} - f(re^{i\theta}) \lvert < \ep \   {\rm for \  any}  \  {1-} \frac{1}{n} < r < 1  \} .$$
Note that $X_n \subset X_{n+1}$ and $X= \bigcup \limits_n X_n$.  Also $m(\partial V \cap X) = m(\partial V \cap \partial \Delta) >0$ by Lemma \ref{hh}  and the fact that $X$ has full measure. Hence $m(\partial V \cap X_{n_0} ) >0$ for $n_0$ large enough. Letting $r_0=  {1-} \frac{1} {n_0}$  and $E_0:=  \partial V \cap X_{n_0}$ we are done. 
\end{proof}

Let $X \subset \partial \Delta$ now denote full measure set where the holomorphic embedding $\phi :\Delta \to \Omega \subset \C^N$ has radial limits. Applying  Lemma \ref{inf}  to the component functions of $\phi$ and fixing $\ep <\delta/N$ we can find a $E_0 \subset  \partial \Delta$ and  an $r_0>0$ such that 
\begin{equation}
\lvert \hat \phi(e^{i\theta}) - \phi(re^{i\theta}) \lvert  \le \delta
\end{equation} 
for all $r_0<r<1$ and $e^{i\theta} \in E_0$. \vspace{3mm}

Let 

$\bullet$ $x \in E_0$ be a point of density of $E_0$,

$\bullet$ $ R : \Delta - \{0\} \rt \partial \Delta$ be the radial projection map $R(z) = \frac{z}{\vert z \vert},$

$\bullet$ $S^1(r) =\{ z \ : \ \vert  z \vert <1\}$, \  \ $A(r)=\{ z \ : \  r < \vert  z \vert <1 \}$.  \vspace{5mm}

\begin{proof} (of Proposition \ref{main-prop}):  Let $x_n \in V$ be a sequence converging to $x$.  Without loss of generality we can assume that $x_n \in A(r_0)$ where $r_0$ waschosen above. Since $V$ is connected by hypothesis, we can find  $a_n >0$ and an embedded curve $\sigma_n: [0, a_n] \rt V$ with $\sigma_n(0) = x_n$, $y_n:=\sigma_n(a_n)  \in S^1(r_0)$ and $\sigma_n([0,a_n)) \subset  A(r_0)$.  Let $F_n= R(\sigma_n) \subset \partial \Delta$. We note that $F_n$ is a closed arc of $S^1 =\partial \Delta$. (See Figure 1.) \vspace{3mm}

\begin{figure}
  \centering
  \includegraphics[scale=0.35]{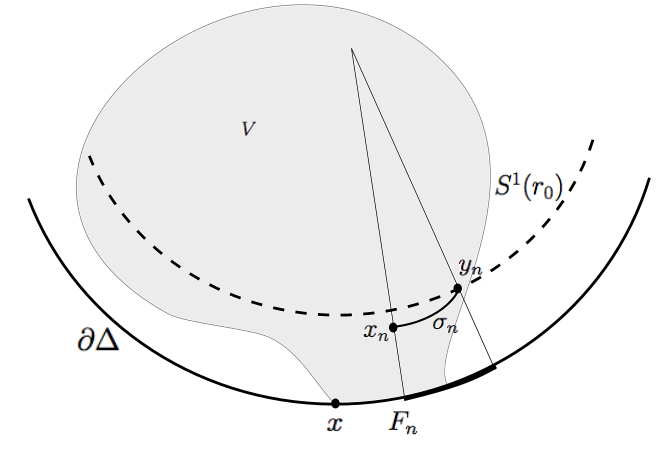}\\
  \caption{Figure for proof of Proposition \ref{main-prop}.}
\end{figure}

{\bf Case 1:} $\inf_n m(F_n)  > 0$.  If we let $F_n =[e^{i\alpha_n}, e^{i\beta_n}]$ these arcs converge to a limiting arc $G= [e^{i\alpha}, e^{i\beta}]$  on $S^1$ of positive length, that is, $m(G)=\beta - \alpha >0$ (by the assumption $\inf (\beta_n - \alpha_n)>0$) . Moreover, since $x_n \to x$, the projections $R(x_n)\in x$ and we have $x\in G$. We can assume, by passing to a  subsequence, that  one of the following hold:

\begin{itemize}

\item $\alpha_{n+1} \le \alpha_n$ and $\beta_{n+1} \le \beta_n$ for all $n$.  In this case $\alpha = \inf \alpha_n$, $\beta = \inf \beta_n$, and we define $G_n=[e^{i\alpha_n}, e^{i\beta}]$.  

\item $\alpha_{n+1} \le \alpha_n$ and $\beta_{n+1} \ge \beta_n$ for all $n$. In this case  $\alpha = \inf \alpha_n$, $\beta = \sup \beta_n$, and we define $G_n=[e^{i\alpha_n}, e^{i\beta_n}]$.  

\item $\alpha_{n+1} \ge \alpha_n$ and $\beta_{n+1} \le \beta_n$ for all $n$.  In this case $\alpha = \sup \alpha_n$, $\beta = \inf \beta_n$, and we define $G_n=[e^{i\alpha_n}, e^{i\beta_n}]$.  

\item $\alpha_{n+1} \ge \alpha_n$ and $\beta_{n+1} \ge \beta_n$ for all $n$.  In this case $\alpha = \sup \alpha_n$, $\beta = \sup \beta_n$, and we define $G_n=[e^{i\alpha}, e^{i\beta_n}]$.  

\end{itemize}

Note that in each of these cases, the subsets $\{G_n\}_{n\geq 1}$ are monotonic, that is, either $G_n \subset G_{n+1}$ for all $n$, or $G_{n+1} \subset G_{n}$ for all $n$.  Moreover, in case they are monotonically increasing, their union $\bigcap\limits_n G_n = G$ and in case they are monotonically decreasing, their intersection $\bigcup\limits_n G_n = G$. We then have 
$$m(G \cap E_0) = \lim \limits_n  m(G_n \cap E_0).$$
 However the density property of $x$ implies that  $m(G \cap E_0)>0$.  Hence $m(G_n \cap E_0)>0$ for $n$ large enough. 
 
 Also, by construction of the sets $G_n$ above, we have $G_n \subset F_n$. 
 
  Let $q \in G_n \cap E_0$: by definition of 
 $F_n$, the radial ray ending at $q$ contains a point $p$ in $\sigma$, in particular a point  in $V \cap A(r_0)$.  Since  $q \in E_0$ we then have 
 $$ \vert \phi(p) - \hat \phi(q) \vert  < \delta$$
 by our choice of $r_0$ above. Since $\phi(p)\in \phi(V) = \om^0_{\leq \delta}$, we obtain that $\phi(q) \in \om^0_{\le 2 \delta}$.  Since this holds for every $q \in G_n \cap E_0$ and $m(G_n \cap E_0)>0$, we have the desired conclusion, with the positive measure set $E_1:= G_n \cap E_0$.
 \vspace{3mm}
 
 {\bf Case 2:} $\inf_n m(F_n)=0$.  This condition implies that a subsequence of the $y_n \in V$ converges to a point $y \in S^1(r_0)$ which lies in the interior of the radial ray through $x$.  Since $x \in E_0$, and $y\in \overline{V}$, we have, as in Case 1, 
 $ \vert \hat \phi(x) - \phi(y) \vert  \leq  \delta$ and $\phi(x) \in \om^0_{\le 2 \delta}$. 
 \vspace{3mm}
 
 If Case 1 holds for even a single $x$ we are done. Otherwise Case 2 holds for  every point of density $x \in E_0$. By Lebesgue's density theorem, points of density $D$ have full measure and setting $E_1 := D \cap E_0$, the proof is complete.  
 
\end{proof}

\subsection{Limits of thick points}

Throughout this section, let $p \in \partial \om$ be a point that is smooth in the sense of Definition \ref{smooth-alex}, with a neighborhood $U = B(p,r) \cap \om$ that is strictly convex (see the setup in the beginning of \S3.1).\vspace{2mm}

We shall now use Proposition \ref{main-prop} to prove that a boundary point $p \in \partial \om$ is an accumulation of ``thick" points in Teichm\"{u}ller space, that is,

\begin{prop}\label{prop1} Fix a compact set $K \subset \M_g$.
There exists a sequence of points $\{p_n\}_{n\geq 1} \subset \Omega$ such that
\begin{itemize}
\item $p_n \to p$ in the Euclidean metric on $\mathbb{C}^N$, and
\item $p_n$ projects to  the compact set $K$ in  $\M_g$.
\end{itemize}
\end{prop}

\begin{proof}
Choose a sequence of neighborhoods  $\{U_n\}_{n\geq 1}$ around $p$ that shrink to $p$ (see Lemma \ref{nest-lem}). It suffices to show that there is a point $p_n \in U_n$ for each $n$, that projects the fixed compact set $K$ in $\M_g$.

Fix an $n\geq 1$. 

By Proposition \ref{key-prop} there exists a complex geodesic  $\phi:\Delta \to \Omega$
with $\phi(0) =q \in \om^0_{\leq \delta}$ such that the set $\Xi \subset \partial \Delta$ of directions   that yield dense Teichm\"{u}ller rays is of full measure in $\partial \Delta$.

By Proposition \ref{main-prop}  for such a complex geodesic $\phi$ there exists a positive measure subset $E_1 \subset \partial \Delta$ of radial directions such that the corresponding geodesic rays $\gamma_\theta$  are eventually contained in $U_n\cap \Omega$.

Then for any $e^{i\theta}$ in the positive measure set $E_1 \cap \Xi$, the Teichm\"{u}ller geodesic ray  $\gamma_\theta$
\begin{itemize}
\item projects to a dense set in $\M_g$, and in particular, recurs to the compact set $K$
\item is also eventually contained in $U_n \cap \om$.
\end{itemize}

Hence there is such a point $p_n\in U_n$ lying along the ray (and in fact infinitely many of them) that projects to the compact set $K$.  
\end{proof}

\subsection{Orbit accumulation points}
Proposition \ref{mcg} is now an immediate consequence of Proposition \ref{prop1} : \vspace{3mm}

{\it Proof of Proposition \ref{mcg}}.  Fix a compact set $K$ in $\M_g$. By Proposition \ref{prop1} there is a sequence of points  $\{p_n\} \subset \Omega$ converging to $p$, such that the projection of $p_n$ to $\M_g$ lies in $K$.
Since $\mathrm{MCG}(S) =  \mathrm{Aut}(\om)$ there exist automorphisms $\gamma_n $ of $\om$ such that $x_n:=\gamma_n(p_n)$ lie in a fixed compact lift $K_0 \subset \Omega$ of $K$. We extract a convergent subsequence $x_n \rt x \in K_0$ and claim that $\gamma_n^{-1}(x) \rt p$. This essentially follows from two facts: on any bounded domain, the Euclidean distance $d^e$ is bounded above by the Kobayashi distance $d^K$  (see (3) of Proposition \ref{koba}) and the two distances induce the same topology:
$$ d^e( \gamma_n^{-1}(x), \gamma_n^{-1}(x_n)) \le  Cd^K( \gamma_n^{-1}(x), \gamma_n^{-1}(x_n))  = Cd^K(x_n,x) \rt 0$$
where $C$ is a constant as in Proposition \ref{koba}, and  the last equality follows from the fact that any biholomorphic automorphism is also an isometry of $\om$  in the Kobayashi metric.
Hence we have:
$$ d^e( \gamma_n^{-1}(x), p)  \le d^e( \gamma_n^{-1}(x), \gamma_n^{-1}(x_n)) + d^e( \gamma_n^{-1}(x_n), p) \rt 0.$$
which proves the claim. \hfill $\square$

\section{Proofs of Theorem \ref{thm1} and \ref{thm3}}

Let $\om \subset \C^{3g-3}$ be as in Theorem \ref{thm1}.

W choose a locally strictly convex boundary point $p\in \partial \om$  that is also  Alexandroff smooth (such a point exists by Alexandroff's theorem - see Corollary \ref{cor-smooth}).

By Proposition \ref{mcg} the boundary point $p$ is an accumulation point of a mapping class group orbit of some point $q\in \om$. In particular, there is a sequence of biholomorphic automorphisms $\phi_j:\om \to \om$ such that $\phi_j(q) \to p$  (in the Euclidean metric) as $j \to \infty$.

We then obtain a sequence of rescalings as in \S2.4. Applying Proposition \ref{rescale}, we  then obtain a domain biholomorphic to $\T$ that has a \textit{continuous} family of automorphisms, which contradicts Royden's theorem that $Aut(\T)$ is discrete(Corollary \ref{cor-royd}).

This completes the proof of Theorem \ref{thm1}.

\subsection{Smooth boundary: Proof of Theorem \ref{thm3}}

In this section let $\Omega$ be a smooth $C^2$-bounded domain.  Theorem \ref{thm3} asserts that a finite-dimensional Teichm\"{u}ller space $\T$ cannot be biholomorphic to $\Omega$.

This has been observed earlier (see pg. 328 of \cite{Yau-survey}); here we give a simpler proof based on the considerations in this paper.

 We recall the following well-known and elementary fact from differential geometry:

\begin{lem} For any bounded domain $\Omega \subset \C^N$ with $C^2$-smooth boundary,  there exists a locally strictly convex boundary point $p \in \partial \Omega$ .
\end{lem}
\begin{proof}
Consider a point $p\in \overline{\Omega}$ that realizes the largest distance from the origin $sup \{\vert x \vert \ : \ x \in \overline \om \} =d$.  It is clear that $p \in \partial \om$ and that $\om \subset B(0,d)$. Moreover, $\overline{\om}$ and $\partial B(0, d)$ intersect at $p$ where they share a common tangent space.  We claim that 
$\om$ is locally strictly convex at $p$.  To see this, let $\rho: \C^N \rt \R$ be a $C^2$-smooth defining function for $\om$, i.e.,
$\rho^{-1}(-\infty, 0) = \om, \ \rho^{-1}(0) = \partial \om$ and $\nabla \rho (p) \neq 0$ for every $p \in \partial \om$. It is then enough to show that the Hessian of $\rho$ at $p$ restricted to $T_p(\partial \om)$  is positive definite, which the following calculation shows:

Let $v \in T_p(\partial \om)$ and $\sigma:(-\ep, \ep) \rt \partial \om$ a $C^2$ curve with $\sigma(0)=p$ and $\sigma'(0)=v$.  Consider $\phi:= \Vert  \sigma \Vert^2:(-\ep, \ep) \rt \R$. Since $t=0$ is a local minimum for
$\phi$, we have
\begin{equation}\label{one}
\langle \sigma''(0), p \rangle \le - \Vert v \Vert ^2.
\end{equation}
On the other hand, differentiating the equation $\rho \circ \sigma (t) =1$ twice at $t=0$ we have
\begin{equation}\label{two}
 Hess \ \rho (v,v)= -\langle \nabla \rho (p), \sigma''(0) \rangle.
\end{equation}
Since $ \nabla \rho (p) =cp$ for some $c>0$, (\ref{one}) and (\ref{two}) together give $ Hess \ \rho (v,v) \ge c \Vert v \Vert^2$.
\end{proof}

An application of Theorem \ref{thm1} completes the proof of Theorem \ref{thm3}.

\appendix 

\section{}

\subsection{ Alexandroff smoothness}

Let $\om \subset \R^n$ be a domain. 

\begin{defn}\label{smooth-alex} We say that a boundary point $p \in \partial \om$ is \textit{Alexandroff smooth} if 
\vspace{2mm}

(i)  $\om$ is locally convex at $p$.
\vspace{2mm}

(ii) there exists $r>0$ such that  $\om \cap B(p,r)$ is convex and $\partial \om\cap B(p,r)$ is the graph of a convex function $\psi:U \cap V \to \mathbb{R}_+$ that has a second order Taylor expansion at $p$. That is, if we assume without loss of generality that $p =0$ and $V = \{x_n =0\}$ is a supporting hyperplane for $\om  \cap B(p,r)$, we have:
\begin{equation}\label{taylor}
\psi(x_1,x_2,\ldots x_{n})  =  \frac{1}{2} \sum\limits_{i,j} H_{i,j}x_ix_j  + o(\lVert x\rVert^2)
\end{equation}
for some $n\times n$ symmetric matrix $H$ (which, for a genuine $C^2$-function, is the Hessian).
\end{defn}

\begin{thm}\label{alex-thm} (Alexandroff  \cite{Alex}, \cite{convexdiff})
If $\om$ is a convex domain, then almost every boundary point is smooth in the above sense. 
\end{thm}
 \vspace{2mm}


\begin{lem}[Interior sphere]\label{sphere} Let  $\om \subset \R^{n}$ be a  domain, and $p\in \partial\om$ be an Alexandroff smooth point. Then $p$ has an interior sphere contact, namely there is a round sphere $S$ contained in $\bar{\om}$ such that $S \cap \bar{\om} = \{p\}$.
\end{lem}
\begin{proof}
We apply Alexandroff's theorem stated above to the convex domain $\om  \cap B(p,r)$.
As in Definition \ref{smooth-alex}, we can assume that $p =0$ and $V = \{x_n =0\}$ is a supporting hyperplane for $\om  \cap B(p,r)$.
Let $\psi:U \cap V \to \mathbb{R}_+$ be a convex defining function as earlier.  
Then it is then easy to check that for any $0< \ep < \min\{r, \frac{1}{2\lVert H\rVert}\}$,  the sphere $S$ centered at $(0,0,\ldots, 0,  \ep)$ and radius $\ep$ lies above the graph defined by (\ref{taylor}) and  contained in $\om \cap B(p,r) $.  \end{proof}

As a consequence, one has:


\begin{cor}\label{cor-smooth} Let $\om \subset \R^n$ be a bounded domain that is locally convex at a point $p \in \partial \om$. Then there is a (possibly different) boundary point $p^\prime \in \partial \om$ that is both locally convex and Alexandroff smooth, and in particular, an interior sphere contact point.
\end{cor}
\begin{proof}
By definition of locally  convex, there is a neighborhood $B(p,r)$ for $r>0$ such that $B(p,r) \cap \om$ is a convex domain. By Theorem \ref{alex-thm} the subset $\partial \om \cap B(p,r)$ contains an Alexandroff smooth point. (All such points are also locally convex.)  As the preceding lemma then asserts, such a point would be a point of contact of an interior sphere in $\om \cap B(p,r) \subset \om$. 
\end{proof}

\subsection{Kim-Krantz-Pinchuk rescaling} 

Let $\om$ be a bounded domain and $q \in  \partial \om$ an Alexandroff smooth point.

Let $\{p_j \}_{j\geq 1} \in D$ be a sequence converging to $q$. For each $j$ let $q_j \in \partial \om$ be the closest boundary point, with
$$ \Vert p_j-q_j \Vert = \inf_{x \in \partial \om} \Vert p_j -x \Vert.$$
We can find a unitary transformation $T_j: \C^N \rt \C^N$ such that the affine map $\psi_j: \C^N \rt \C^N$ defined by $\psi_j(z)=T_j(z-p_j)$ satisfies
$$\psi_j(\om) \subset \{(z_1,...,z_n) \ : \ Re(z_n)>0\}.$$
Denote by $V_{n-1}^j$ the orthogonal complement in $\C^N$ of the line joining the origin and $\psi_j(q_j)$ and let $D^j_{n-1}$ be the projected slice
$$D_{n-1}^j= \{z \in V^j_{n-1} \ :  z+ \psi_j(q_j) \in \psi_j(\om)\}. $$
Note that $D_{n-1}^j$ is a domain in $V_{n-1}^j$ containing the origin. Let $x^j_{n-1}$ be a point in $\partial D^j_{n-1}$ closest to the origin. Let $V^j_{n-2}$ denote the orthogonal complement in $V_{n-1}^j$ of the line spanned by $x_{n-1}^j$ and let
$$D^j_{n-2}=D^j_{n-1} \cap V_{n-2}^j.$$  We continue this process as long as it is possible to do so.

 By this process we obtain mutually orthogonal vectors $x_1^j,...,x_{n-1}^j$. Then the vectors $$e^j_l= \frac{x_l^j}{\Vert x_l^j\Vert}$$
   form an orthonormal basis for $\C^N$. For each $j$ define the complex linear mapping
   $$\Lambda_l^j(x_l^j)=  e_l^j, \ \ \ \ \ l=1,2,...,n.$$

   \begin{defn}
The {\it Pinchuk rescaling sequence} associated to $\{p_j\}_{j\geq1}$ is then the sequence of complex linear maps
$$\sigma_j : = \Lambda_j \circ \psi_j: \C^N \rt \C^N$$

\end{defn}

Now assume that  $\{p_j\}_{j\geq1}$ is an automorphism orbit, that is, $p_j= \phi_j(p_0)$ for some $p_0 \in \om$ and $\phi_j \in Aut(\om)$. Let
$$ \omega_j = \sigma_j \circ \phi_j: \om  \rt \C^N$$
be the resulting sequence of biholomorphisms.

 One then has the following

\begin{proposition}\label{rescale} In the above set-up, and in particular under the assumption that the initial bounded domain $\om$ is locally convex at the smooth boundary point $q$, the following hold:

\begin{enumerate}

\item  A subsequence of $\{\omega_j\}_{j\geq 1}$ converges uniformly on compact subsets to a holomorphic embedding $\hat \omega: \om \rt \C^N$. \vspace{2mm}

\item For the above subsequence $\omega_j (\om)$ converges, in the local Hausdorff topology, to $\hat \omega (\om)$. \vspace{2mm}

\item  Let $p$ be a smooth boundary point with an interior sphere contact  (as in Lemma \ref{sphere}) . Then the $1$-dimensional slices
$$\Sigma_j = \{z \in \om \ :  \ z-q_j = \lambda(p_j -q_j) \} $$ converge in the local Hausdorff topology to the upper-half plane
$\{z \in \C^N \ : \ Re(z_n) \ge 0 \}$. In particular the map $z \mapsto z+te_n$ is an automorphism of $\hat \omega (\om)$ for each $t \in \R$.

\end{enumerate}

\end{proposition}

For the proof, we refer to \S4 of \cite{KimKrantz1} and  \S2 of \cite{Kim-Bers} (see also \cite{KimKrantz2}).

\bibliographystyle{amsalpha}
\bibliography{Convex-refs}

\providecommand{\bysame}{\leavevmode\hbox to3em{\hrulefill}\thinspace}
\providecommand{\MR}{\relax\ifhmode\unskip\space\fi MR }
\providecommand{\MRhref}[2]{%
  \href{http://www.ams.org/mathscinet-getitem?mr=#1}{#2}
}
\providecommand{\href}[2]{#2}
\begin{thebibliography}{McM91}

\bibitem[Ahl06]{Ahlfors}
Lars~V. Ahlfors, \emph{Lectures on quasiconformal mappings}, second ed.,
  University Lecture Series, vol.~38, American Mathematical Society,
  Providence, RI, 2006, With supplemental chapters by C. J. Earle, I. Kra, M.
  Shishikura and J. H. Hubbard.

\bibitem[Ale39]{Alex}
A.~D. Alexandroff, \emph{Almost everywhere existence of the second differential
  of a convex function and some properties of convex surfaces connected with
  it}, Leningrad State Univ. Annals [Uchenye Zapiski] Math. Ser. \textbf{6}
  (1939), 3--35. \MR{0003051}

\bibitem[BCP96]{convexdiff}
Gabriele Bianchi, Andrea Colesanti, and Carlo Pucci, \emph{On the second
  differentiability of convex surfaces}, Geom. Dedicata \textbf{60} (1996),
  no.~1, 39--48.

\bibitem[Ber60]{Bers}
Lipman Bers, \emph{Spaces of {R}iemann surfaces as bounded domains}, Bull.
  Amer. Math. Soc. \textbf{66} (1960), 98--103.

\bibitem[BF14]{BrannerBook}
Bodil Branner and N\'uria Fagella, \emph{Quasiconformal surgery in holomorphic
  dynamics}, Cambridge Studies in Advanced Mathematics, vol. 141, Cambridge
  University Press, Cambridge, 2014, With contributions by Xavier Buff, Shaun
  Bullett, Adam L. Epstein, Peter Ha\"\i ssinsky, Christian Henriksen, Carsten
  L. Petersen, Kevin M. Pilgrim, Tan Lei and Michael Yampolsky.

\bibitem[Fra89]{Frankel}
Sidney Frankel, \emph{Complex geometry of convex domains that cover varieties},
  Acta Math. \textbf{163} (1989), no.~1-2, 109--149.

\bibitem[Hub06]{Hubbard}
John~Hamal Hubbard, \emph{Teichm\"uller theory and applications to geometry,
  topology, and dynamics. {V}ol. 1}, Matrix Editions, Ithaca, NY, 2006.

\bibitem[Kim04]{Kim-Bers}
Kang-Tae Kim, \emph{On the automorphism groups of convex domains in {$\Bbb
  C^n$}}, Adv. Geom. \textbf{4} (2004), no.~1, 33--40.

\bibitem[KK03]{KimKrantz1}
Kang-Tae Kim and Steven~G. Krantz, \emph{Some new results on domains in complex
  space with non-compact automorphism group}, J. Math. Anal. Appl. \textbf{281}
  (2003), no.~2, 417--424.

\bibitem[KK08]{KimKrantz2}
\bysame, \emph{Complex scaling and geometric analysis of several variables},
  Bull. Korean Math. Soc. \textbf{45} (2008), no.~3, 523--561.

\bibitem[Kru91]{Krushkal}
S.~L. Krushkal, \emph{Strengthening pseudoconvexity of finite-dimensional
  {T}eichm\"uller spaces}, Math. Ann. \textbf{290} (1991), no.~4, 681--687.

\bibitem[Leh87]{Lehto}
Olli Lehto, \emph{Univalent functions and {T}eichm\"uller spaces}, Graduate
  Texts in Mathematics, vol. 109, Springer-Verlag, New York, 1987.

\bibitem[Mar]{Mark}
Vlad Markovic, \emph{{C}aratheodory's metric in {T}eichm\"{u}ller spaces and
  {L}-shaped pillowcases}, Preprint.

\bibitem[Mas80]{Masur0}
Howard Masur, \emph{Uniquely ergodic quadratic differentials}, Comment. Math.
  Helv. \textbf{55} (1980), no.~2, 255--266.

\bibitem[Mas82]{Masur}
\bysame, \emph{Interval exchange transformations and measured foliations}, Ann.
  of Math. (2) \textbf{115} (1982), no.~1, 169--200.

\bibitem[McM91]{McM}
Curt McMullen, \emph{Cusps are dense}, Ann. of Math. (2) \textbf{133} (1991),
  no.~1, 217--247. \MR{1087348}

\bibitem[Ran95]{Ransford}
Thomas Ransford, \emph{Potential theory in the complex plane}, London
  Mathematical Society Student Texts, vol.~28, Cambridge University Press,
  Cambridge, 1995.

\bibitem[Ros79]{Rosay}
Jean-Pierre Rosay, \emph{Sur une caract\'erisation de la boule parmi les
  domaines de {${\bf C}^{n}$}\ par son groupe d'automorphismes}, Ann. Inst.
  Fourier (Grenoble) \textbf{29} (1979), no.~4, ix, 91--97. \MR{558590}

\bibitem[Roy71]{Royd}
H.~L. Royden, \emph{Automorphisms and isometries of {T}eichm\"uller space},
  Advances in the {T}heory of {R}iemann {S}urfaces ({P}roc. {C}onf., {S}tony
  {B}rook, {N}.{Y}., 1969), Ann. of Math. Studies, No. 66. Princeton Univ.
  Press, Princeton, N.J., 1971, pp.~369--383. \MR{0288254}

\bibitem[Shi84]{Shiga}
Hiroshige Shiga, \emph{On analytic and geometric properties of {T}eichm\"uller
  spaces}, J. Math. Kyoto Univ. \textbf{24} (1984), no.~3, 441--452.

\bibitem[Siu91]{UnifSiu}
Yum~Tong Siu, \emph{Uniformization in several complex variables}, Contemporary
  geometry, Univ. Ser. Math., Plenum, New York, 1991, pp.~95--130.

\bibitem[Vee82]{Veech}
William~A. Veech, \emph{Gauss measures for transformations on the space of
  interval exchange maps}, Ann. of Math. (2) \textbf{115} (1982), no.~1,
  201--242.

\bibitem[Won77]{Wong}
B.~Wong, \emph{Characterization of the unit ball in {${\bf C}^{n}$} by its
  automorphism group}, Invent. Math. \textbf{41} (1977), no.~3, 253--257.

\bibitem[Yau11]{Yau-survey}
Shing-Tung Yau, \emph{A survey of geometric structure in geometric analysis},
  Surveys in differential geometry. {V}olume {XVI}. {G}eometry of special
  holonomy and related topics, Surv. Differ. Geom., vol.~16, Int. Press,
  Somerville, MA, 2011, pp.~325--347.

\bibitem[Yeu03]{Yeung}
Sai-Kee Yeung, \emph{Bounded smooth strictly plurisubharmonic exhaustion
  functions on {T}eichm\"uller spaces}, Math. Res. Lett. \textbf{10} (2003),
  no.~2-3, 391--400.

\end{thebibliography}

\end{document}